\documentclass[12pt]{amsart}

\usepackage{amssymb,epsf}

\newtheorem{thm}{Theorem}[section] 

\newtheorem{cor}[thm]{Corollary}

\newtheorem{lem}[thm]{Lemma}
\newtheorem{exa}[thm]{Example}

\newtheorem{algo}[thm]{Algorithm}

\newcommand{\R}{\mathbb R}

\begin{document}

\title[Convex hull and the separating matrix]{Towards the computation of the convex hull of a configuration from its corresponding separating matrix}
\author{Elie Feder and David Garber}
\address{Kingsborough Community College of CUNY, Dept. of Mathematics and Computer Science, 2001 
Oriental Blvd., Brooklyn, NY 11235, USA}
\email{davball@aol.com, efeder@kbcc.cuny.edu}

\address{Department of Applied Mathematics, Holon Institute of Technology, 52 Golomb St., PO Box 305, 58102 Holon, Israel}
\email{garber@hit.ac.il}

\begin{abstract}
In this paper, we cope with the following problem: compute the size of the 
convex hull of a configuration $\mathcal C$, where the given data is the number of 
separating lines between any two points of the configuration (where the lines are generated 
by pairs of other points of the configuration). 

We give an algorithm for the case that the convex hull is of size $3$, and a partial algorithm 
and some directions for the case that the convex hull is of size bigger than $3$.  
\end{abstract}

\maketitle

\section{Introduction}
A finite set $\mathcal P = \{P_1,\cdots, P_n\}$ of $n$ points
in the oriented affine plane $\R^2$ is a {\it configuration in
general position} if three points in $\mathcal P$ are never collinear.
Two configurations of $n$ points in general position $\mathcal P^1$ and
$\mathcal P^2$ are {\it isotopic} if they can be joined by a continuous
path of configurations in general position.

A line $L \subset \R^2$ {\it separates} two points $P,Q \in \R^2\setminus L$
if $P$ and $Q$ are in different connected components of $\R^2 \setminus L$.
Given a configuration in general position $\mathcal P$, we denote by $n(P,Q)$
the number of separating lines defined by pairs of points in $\mathcal P
\setminus \{P,Q\}$.

Given a configuration in general position $\mathcal P$, we define
{\it the separating matrix of a configuration}, $S(\mathcal P)$,
to be the symmetric matrix of order $n$, defined by:
$$(S(\mathcal P))_{ij}=n(P_i,P_j)$$

The interesting question is which data we can retrieve
from this separating matrix. In this paper, we introduce an algorithm
that partially computes the convex hull from the separating matrix.

\medskip

There are matrices associated to planar configurations of points which determine 
the configurations, for example the $\lambda$-matrix which was defined by Goodman and 
Pollack \cite{GP}. An interesting question is to study the connection 
between these two matrices.

\medskip

The paper is organized as follows. In Section \ref{observations}, we give some simple 
observations regarding the separating matrix, which yield some restriction on this matrix.
In Section \ref{size3} we present an algorithm for computing the convex hull 
of a configuration, in case that its size is $3$. 
In Section \ref{general}, we give partial algorithm for computing the convex hull in case 
that its size is bigger than $3$. We also give possible directions for solving this problem.

\section{Simple observations about the separating matrices associated to configurations}\label{observations}

In this section, we point out some properties of the separating matrix.

\begin{lem} Given a separating matrix which represents a
configuration of $n$ points. Then the maximal entry of this matrix
is ${n-2 \choose 2}$.
\end{lem}

\begin{proof}
The separating lines are generated from pairs of the remaining points. 
There are $n-2$ such points, and hence there are ${n-2 \choose 2}$ such pairs.  
\end{proof}

The next point is how many odd and even entries we have in this matrix.
For this we recall the Orchard relation (see \cite{B} and \cite{BG1}). Given a planar 
configuration of $n$ points, we say that two points $P,Q$ are equivalent if
$n(P,Q) \equiv (n-1) \pmod 2$. We have shown that this is an equivalence relation 
with at most two equivalence classes \cite{B,BG1}.   

\begin{lem}

\begin{enumerate}
\item If $n$ is even, then there are $i(n-i)$ even entries in the 
upper triangle part (excluding the diagonal) of the separating matrix 
for some $0 \leq i \leq n$. The rest of the entries are odd.
\item If $n$ is odd, then there are $i(n-i)$ odd entries in the 
upper triangle part (excluding the diagonal) of the separating matrix 
for some $0 \leq i \leq n$. The rest of the entries are even.
\end{enumerate}

\end{lem}

\begin{proof}
Assume $n$ is even. By the definition of the Orchard relation, 
two points in different equivalence classes have an even number of separating lines. 
If one class has $i$ points, then the other has $n-i$ points, thus there are $i(n-i)$ 
pairs of points in diffferent classes, and therefore, $i(n-i)$ even entries. All the 
rest correspond to pairs of points from the same equivalence class, and hence have
odd entries.

For the case that $n$ is odd, the proof is identical.
\end{proof}

\begin{exa}
For $n=6$, the upper triangular part (excluding the diagonal) has 15 entries. 
The options for the number of even entries are: 
$6 \cdot (6-6)=0, 5 \cdot (6-5)=5, 4 \cdot (6-4)=8, 3 \cdot (6-3)=9$.
\end{exa}

One more check, based on the Orchard relation, is the following:

\begin{lem}
Let $S(\mathcal P)$ be the separating matrix of a configuration $P$.
\begin{enumerate}
\item If $(S(\mathcal P))_{ij} \equiv (S(\mathcal P))_{jk} \pmod 2$,  
then  $(S(\mathcal P))_{ik} \equiv (n-1) \pmod 2$.
\item If  $(S(\mathcal P))_{ij} \not\equiv (S(\mathcal P))_{jk} \pmod 2$, 
then $(S(\mathcal P))_{ik} \equiv n \pmod 2$
\end{enumerate}

\end{lem}

\begin{proof} 
(1) If $(S(\mathcal P))_{ij} \equiv (S(\mathcal P))_{jk} \equiv n \pmod 2$, 
then $P_i$ and $P_j$ are in different classes, and $P_j$ and $P_k$ are in different classes. 
This implies that $P_i$ and $P_k$ are in the same class. 

On the other hand, if $(S(\mathcal P))_{ij} \equiv (S(\mathcal P))_{jk} \equiv (n-1) \pmod 2$, 
then $P_i$ and $P_j$ are in the same class, and $P_j$ and $P_k$ are in the same class. 
Therefore, $P_i$ and $P_k$ are in the same class too. 
 
(2) If  $(S(\mathcal P))_{ij} \not\equiv (S(\mathcal P))_{jk} \pmod 2$, then we have two cases:
\begin{enumerate}
\item[(a)] $P_i$ and $P_j$ are in the same class, while $P_j$ and $P_k$ are 
in different classes.
\item[(b)] $P_i$ and $P_j$ are in different classes, while $P_j$ and $P_k$ are 
in the same class.
\end{enumerate} 
In both cases, we can conclude that $P_i$ and $P_k$ are in different classes, 
and the result follows.
\end{proof}

Hence, we have the following corollary:
\begin{cor}
Necessary conditions for a matrix to be a separating matrix of a
configuration are:
\begin{enumerate}
\item It should be  symmetric with diagonal $0$.

\item All entries should be smaller (or equal) than ${n-2 \choose
2}$

\item The matrix should satisfy the conditions of the last two lemmas.
\end{enumerate}
\end{cor}

\section{Convex hull of size $3$}\label{size3}

We start with the simplest case, where the size of the convex hull is $3$.
We present the algorithm for the case of convex hull of size $3$. 

\begin{algo}
\begin{enumerate}
\item Choose $i,j,k$ such that $1 \leq i<j<k \leq n$
\item If $(S(\mathcal P))_{ij} + (S(\mathcal P))_{ik} + (S(\mathcal P))_{jk} = n^2 -4n+3$, then return: ''Convex hull is of size $3$ and it is $P_i,P_j,P_k$''.
\item If for all $1 \leq i<j<k \leq n$, $(S(\mathcal P))_{ij} + (S(\mathcal P))_{ik} + (S(\mathcal P))_{jk} \neq n^2 -4n+3$, then 
return: ''Convex hull of size larger than $3$''.
\end{enumerate}
\end{algo}
 
\subsection{The correctness of the algorithm}
The correctness of the algorithm is based on the following lemmas.

Let us assume we have a configuration with convex hull points $P_i, P_j, P_k$.
In our matrix $(S(\mathcal P))_{ij} = n(P_i,P_j)$ is the number of lines separating
$P_i$ and $P_j$, for all $i$ and $j$. 

First, we will show that if the convex hull is of size $3$, then the sum of separating lines 
on the convex hull is indeed $n^2 -4n+3$:  

\begin{lem}
Let $C$ be a configuration of $n$ points. Assume that the points $P_i, P_j, P_k$ form its 
convex hull. Then:  
$$(S(\mathcal P))_{ij} + (S(\mathcal P))_{ik} + (S(\mathcal P))_{jk} = n^2 -4n+3.$$ 
\end{lem}

\begin{proof}
We have $n-3$ internal points (which are not on the convex hull). Each one of them 
contributes $3$ separating lines on the convex hull (by the lines generated by the internal point 
and the three points of the convex hull). Moreover, each pair of them 
contributes $2$ to the number of separating lines 
(by the line generated by this pair of points).
Hence we have the following number of separating lines on the convex hull:
$$3(n-3)+2 \cdot \; {(n-3)(n-4) \over 2}=(n-3)(n-1)=n^2-4n+3$$
as needed.
\end{proof}

Now, we will show that for any triple of points $P_i,P_j,P_k$, which is not 
the convex hull, we have:  
$(S(\mathcal P))_{ij} + (S(\mathcal P))_{ik} + (S(\mathcal P))_{jk} < n^2 -4n+3$.

\begin{lem}\label{non-ch-position-lemma}
Let $C$ be a configuration of $n$ points with a convex hull of size $3$. 
Let $P_i, P_j, P_k$ be a triple of points which is not the convex hull.
Then:  
$(S(\mathcal P))_{ij} + (S(\mathcal P))_{ik} + (S(\mathcal P))_{jk} < n^2 -4n+3$. 
\end{lem}

\begin{proof} 
First, notice that from the previous lemma we can derive that 
$n^2-4n+3$ is the highest possible amount of separating lines
for a boundary of any triangle.
Now, consider a triple of points $P_i,P_j,P_k$ which are not the convex
hull of the configuration. Hence, there is a point $P$ which is outside the triangle 
generated by $P_i,P_j,P_k$. A simple observation shows that at least one of the 
lines generated by $P$ and the one of the points $P_i,P_j,P_k$ does not cross 
the boundary of the triangle, and hence the total number of separating lines 
is strictly smaller than $n^2-4n+3$.   
\end{proof}

By similar argument to that of the previous lemma, we have: 
\begin{lem}
Let $C$ be a configuration of $n$ points with a convex hull of size larger than $3$. 
Let $P_i, P_j, P_k$ be a triple of points.
Then:  
$$(S(\mathcal P))_{ij} + (S(\mathcal P))_{ik} + (S(\mathcal P))_{jk} < n^2 -4n+3.$$ 
\ \hfill $\qed$
\end{lem}

Hence, for determining if a given matrix corresponds to a configuration whose size 
of its
convex hull is $3$, we have to do the following: For each  triple of indices $i, j, k$, 
compute $(S(\mathcal P))_{ij} + (S(\mathcal P))_{ik} +(S(\mathcal P))_{jk}$. 
If for a triple $i,j,k$, we get the maximal value $n^2-4n+3$, 
then the configuration has a convex hull of size $3$,
which consists of the points $P_i, P_j$, and $P_k$. 
If for all triples of indices $i, j, k$, 
we get $$(S(\mathcal P))_{ij} + (S(\mathcal P))_{ik} +(S(\mathcal P))_{jk}<n^2-4n+3,$$ 
then the configuration has a convex hull of size bigger than $3$.

\subsection{Complexity}

It is easy to see that the complexity is O($n^3$), 
since we have to check ${n \choose 3}$ triples of points.

\section{The general case}\label{general}

For the general case, we have some partial results.

We start with the expected number of separating lines on the convex hull of size $k$.

\begin{lem}
Let $C$ be a configuration of $n$ points. Assume that the points 
$P_{i_1}, P_{i_2}, \cdots , P_{i_k}$ form its convex hull in this order. 
Then:  
$$(S(\mathcal P))_{i_1 i_2} + (S(\mathcal P))_{i_2 i_3} + \cdots + 
(S(\mathcal P))_{i_{k-1} i_k} + (S(\mathcal P))_{i_k i_1} = n^2 -(k+1)n+k.$$
\end{lem}

\begin{proof}
We have $n-k$ internal points (which are not on the convex hull). Each one of them 
contributes $k$ separating lines on the convex hull (by the lines generated by the internal point 
and the $k$ points of the convex hull, see Figure \ref{intersect_hull}, Line (1)). 
Moreover, Each pair of them 
contributes $2$ to the number of separating lines 
(by the line generated by this pair of points, 
see Figure \ref{intersect_hull}, Line (2)).

\begin{figure}[h]
\epsfysize=3cm
\centerline{\epsfbox{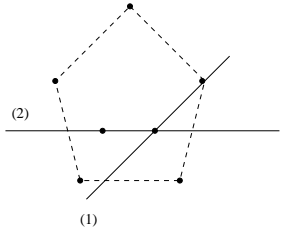}}
\caption{Examples for lines intersecting the convex hull}\label{intersect_hull}
\end{figure}

Hence we have the following total number of separating lines on the convex hull:
$$k(n-k)+2 \cdot \; {(n-k)(n-k-1) \over 2}=(n-k)(n-1)=n^2-(k-1)n+k$$
as needed.
\end{proof}

\begin{lem}
Let $C$ be a configuration of $n$ points. Assume that the points 
$P_{i_1}, P_{i_2}, \cdots , P_{i_k}$ form its convex hull in this order. 
Let $\sigma$ be a permutation in $S_k$ (the symmetric group on $k$ elements).
Then: 
$$\min_{\sigma\in S_k} \left( S(\mathcal P))_{i_{\sigma(1)} i_{\sigma(2)}} + \cdots + 
(S(\mathcal P))_{i_{\sigma(k-1)} i_{\sigma(k)}} + 
(S(\mathcal P))_{i_{\sigma(k)} i_{\sigma(1)}} \right) = n^2 -(k+1)n+k$$
\end{lem}

\begin{proof}
We will compare the number of hits on paths going through these $k$ points exactly once. 
We will show that the number of hits on the convex hull is {\it strictly smaller} than 
the number of hits on a path going through these $k$ points exactly once, which is not 
the convex hull.

\medskip

We have 3 classes of lines:
\begin{enumerate}
\item  Lines determined by one internal point and one point on the convex hull 
(see Line (1) in Figure \ref{intersect_diag}): Such a line intersects the cycle of the convex 
hull exactly once. For any other $k$-cycle, each of these lines hits at least once.
\item Lines determined by two internal points: For the hull, each of these lines 
intersects the hull exactly twice (see Line (2) in Figure \ref{intersect_diag}). 
For any other k-cycle, each of these lines hits at least twice.  
\item Lines determined by two points of the convex hull: These lines do not intersect 
the convex hull, since the convex hull is not self-intersecting. 
But any other cycle is self-intersecting 
and hence these lines will contribute (see the dotted lines in Figure \ref{intersect_diag}).
This contribution yields the ``strictly smaller'' part. 
\end{enumerate}

\begin{figure}[h]
\epsfysize=3cm
\centerline{\epsfbox{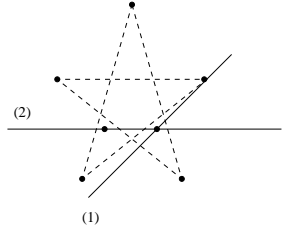}}
\caption{Examples for lines intersecting the hamiltonian cycle which is not convex}\label{intersect_diag}
\end{figure}

Hence we are done.
\end{proof}

Now we will show that if a configuration has a convex hull of size $k$, 
any other $k$ points {\it in convex position} will have less separating lines over 
its convex hull. 

\begin{lem}
Let $\mathcal P$ be a configuration of $n$ points whose convex hull is of size $k$.
Let $P_{i_1},\cdots ,P_{i_k}$ be $k$ points {\bf in convex position} (in this order), which 
do not form the convex hull of $P$. Then: 
$$(S(\mathcal P))_{i_1 i_2} + \cdots + 
(S(\mathcal P))_{i_{k-1} i_k} + (S(\mathcal P))_{i_k i_1} < n^2 -(k+1)n+k.$$ 
\end{lem}

The proof of this lemma uses the same argument as the proof 
of Lemma \ref{non-ch-position-lemma}.

Similarly, it is easy to see the following:

\begin{lem}
Let $\mathcal P$ be a configuration of $n$ points whose convex hull is of size $k$.
Let $3 \leq m \leq n$, $m \not =k$. Let $P_{i_1},\cdots ,P_{i_m}$ be $m$ points in convex 
position (in this order). Then: 
$$(S(\mathcal P))_{i_1 i_2} + \cdots + 
(S(\mathcal P))_{i_{m-1} i_m} + (S(\mathcal P))_{i_m i_1} < n^2 -(m+1)n+m.$$ 
\end{lem}

Based on these lemmas, one can try to compute the convex hull from the separating
matrix by using a similar algorithm to the case of convex hull of size $3$: 

\begin{algo}
\begin{enumerate}
\item Set k:=3.
\item Choose $i_1,i_2,\cdots, i_k$ such that $1 \leq i_1<i_2<\cdots <i_k \leq n$
\item If $$\min_{\sigma\in S_k} \left( S(\mathcal P))_{i_{\sigma(1)} i_{\sigma(2)}} + \cdots + 
(S(\mathcal P))_{i_{\sigma(k-1)} i_{\sigma(k)}} + 
(S(\mathcal P))_{i_{\sigma(k)} i_{\sigma(1)}} \right) = n^2 -(k+1)n+k,$$
 then return: ''Convex hull is of size $k$ and it is $P_{i_1},P_{i_2},\cdots,P_{i_k}$''.
\item If for all $1 \leq i_1<i_2< \cdots < i_k \leq n$, the condition in (3) is not
satisfied, then $k:=k+1$ and return to Step (2).
\end{enumerate}
\end{algo}

For covering all the subsets of $k$ points out of the $n$ points of the configuration,
we have used the implementation of \cite{LHS}.
 
The problem of this algorithm is that it found ``fake'' convex hulls, i.e. 
one can have $m$ points ($m \neq k$) NOT in convex position which still yield 
the correct minimal number of separating lines on its ``convex hull''.

For example, in Figure \ref{conf_7_4}, the real convex hull of size $5$ is dashed 
and the $4$ black points are the ``fake'' convex hull (notice that they 
are not in convex position). 

\begin{figure}[h]
\epsfysize=8cm
\centerline{\epsfbox{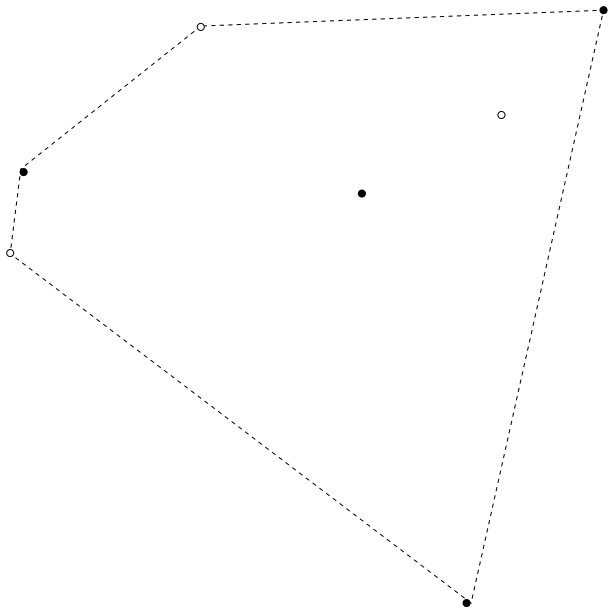}}
\caption{A configuration with a ``fake'' convex hull}\label{conf_7_4}
\end{figure}

Out of 135 configurations of 7 points in general position, one gets 13 such ``fake'' 
convex hulls (for less than 7 points, there were no ``fake'' convex hulls). 
We have used the databse of Aichholzer and Krasser (see \cite{AK}).
 
\medskip

So, our next aim is to find how can we outrule the fake convex hulls and keep 
only the real ones. 
 
One can try the following probabilistic way to rule out these ``fakes''. 
One can imagine that the sum of separating lines over the $n-1$ lines going 
out from a point on the convex hull (i.e. the sum of entries in the row corresponding 
to a point in the convex hull) will be higher than (or at least equal to) 
this corresponding sum for a point not on the convex hull (i.e. the sum of entries 
in the row corresponding to a point not in the convex hull).   

For $n=7$ and $n=8$ points, this check indeed rules out the ``fake'' convex hulls.   
For $n=9$ points, it also rules out the ``fake'' convex hulls, but it also rules out 
the correct convex hull in 28 cases (out of 158817 cases), since 
the corresponding sum of one of the points of the convex hull is strictly smaller 
than this sum for an internal point.

\medskip

One direction for suggesting a different algorithm for this problem is the following: 
once we find out that a case of $k$ points 
that satisfies the correct number of separating lines, but fails to satisfy 
the maximal row sum condition (see two paragraphs above), make an extra check: 
if in this case we have that 
the corresponding sums of the least $m$ rows of the convex hull are equal 
($m$ can be equal to $1$), then check if the corrsponding rows of these $k$ points 
are amongst the highest $k+m$ rows. If so, these $k$ points form the convex hull.

This check will fail if there will be a ``fake'' also here: $k$ points which has the 
correct number of separating lines and satisfy the maximal row sum condition 
(such a ``fake'', if exists, has at least $10$ points). 

\medskip

We now show that the cycle consisting of the points of the real convex hull 
(in order), combined with any number of internal points cannot be considered 
as a ``fake'' convex hull.
 
\begin{lem}
Let $\mathcal P$ be a configuration of $n$ points whose convex hull is of 
size $k$. Let $P_{i_1},\cdots ,P_{i_k}$ be the $k$ points of the convex hull 
(in this order) of $P$. Let $P_{i_m}$ be an internal point. 
Consider the cycle $P_{i_1},\cdots  ,P_{i_k}, P_{i_m}$.
 Then 
$$(S(\mathcal P))_{i_1 i_2} + (S(\mathcal P))_{i_2 i_3} + \cdots + 
(S(\mathcal P))_{i_{k-1} i_k} + (S(\mathcal P))_{i_k i_m} + (S(\mathcal P))_{i_m i_1}
> n^2 -(k+2)n+ (k+1).$$ 
\end{lem}

\begin{proof}
We have $n-k-1$ points which are not on this cycle. Each one of them 
contributes at least $k+1$ separating lines to this cycle 
(by the lines generated by the points on this cycle with those not on this cycle).
Moreover, each pair of them contributes at least $2$ to the number of 
separating lines (by the line generated by this pair of points). 
Additionally, the line connecting
$P_{i_m}$ to $P_{i_k}$, and the line connecting  $P_{i_m}$ to $P_{i_1}$ each 
contribute one separating line to the cycle. 

Hence we have the following number of separating lines on the cycle:
$$(k+1)(n-k-1)+2 \cdot {(n-k-1)(n-k-2) \over 2} + 2=(n-k-1)(n-1)+2=$$
$$=n^2-(k+2)n + (k+3) > n^2-(k+2)n + (k+1) $$ 
as needed.



\end{proof}

The same argument will apply if one adds any number of points to the cycle 
of the convex hull.

\section*{Acknowledgements}
We thank Roland Bacher and Rom Pinchasi for fruitful discussions.

\end{document}